\documentclass[a4paper]{amsart}
\usepackage[latin1]{inputenc}
\usepackage{amssymb}
\usepackage{amsmath}
\usepackage{mathrsfs}
\usepackage{eufrak}
\usepackage{amsthm}
\usepackage{amsfonts}
\usepackage{textcomp}
\usepackage{graphicx}
\usepackage[pdftex]{color}
\usepackage{paralist}
\usepackage[shortlabels]{enumitem}
\usepackage{hyperref}
\usepackage{comment}
\usepackage[arrow, matrix, curve]{xy}

\newtheorem*{corollary*}{Corollary}
\newtheorem*{theorem*}{Theorem}
\newtheorem{theorem}{Theorem}[section]

\newtheorem{lemma}[theorem]{Lemma}

\newtheorem*{claim*}{Claim}
\newtheorem*{conjecture}{Conjecture}

\theoremstyle{definition}

\theoremstyle{remark}

\numberwithin{equation}{theorem}

\makeatletter
\renewcommand*\env@matrix[1][\
arraystretch]{%
  \edef\arraystretch{#1}%
  \hskip -\arraycolsep
  \let\@ifnextchar\new@ifnextchar
  \array{*\c@MaxMatrixCols c}}
\makeatother

\begin{document}

\title{On a conjecture about Morita algebras}
\date{\today}

\subjclass[2010]{Primary 16G10, 16E10}

\keywords{Morita algebras, dominant dimension, tilting modules}

\author{Bernhard B{\"o}hmler }
\address{Institute of algebra and number theory, University of Stuttgart, Pfaffenwaldring 57, 70569 Stuttgart, Germany}
\email{bernhard.boehmler@googlemail.com}
\author{Ren\'{e} Marczinzik}
\address{Institute of algebra and number theory, University of Stuttgart, Pfaffenwaldring 57, 70569 Stuttgart, Germany}
\email{marczire@mathematik.uni-stuttgart.de}

\begin{abstract}
We give an example of a Morita algebra $A$ with a tilting module $T$ such that the algebra $End_A(T)$ has dominant dimension at least two but is not a Morita algebra.
This provides a counterexample to a conjecture by Chen and Xi from \cite{CX}.
\end{abstract}

\maketitle
\section*{Introduction}
In this article we assume that all rings are finite dimensional algebras over a field $K$ and all modules are finitely generated right modules unless stated otherwise.
Recall that the dominant dimension $domdim(M)$ of a module $M$ with minimal injective coresolution $(I^i)$ is defined as zero in case $I^0$ is not projective and $domdim(M):= \sup \{ n \geq 0 | I^i$ is projective for $i=0,1,...,n \}+1$ otherwise. The dominant dimension of an algebra is defined as the dominant dimension of the regular module. It is well known that an algebra has dominant dimension at least one if and only if there is a minimal faithful projective-injective right module $eA$ for some idempotent $e$ of $A$.
All Nakayama algebras have dominant dimension at least one and therefore have a minimal faithful projective-injective module given by the direct sum of all indecomposable projective-injective modules, see for example chapter 32 of \cite{AnFul}.
For more on the dominant dimension we refer to \cite{Yam}.
In \cite{KY} the authors introduced Morita algebras as algebras $A$ that are algebras with dominant dimension at least two and a minimal faithful projective-injective module $eA$ such that $eAe$ is selfinjective. Morita algebras contain several important classes of algebras such as Schur algebras $S(n,r)$ for $n \geq r$ or blocks of category $\mathcal{O}$ and provide a useful generalisation of selfinjective algebras. 
At the end of the article \cite{CX} the authors provided three conjectures related to the dominant dimension of algebras. Their third conjecture states the following:
\begin{conjecture}
Suppose two algebras $A$ and $B$ are derived equivalent. If $A$ is a Morita algebra and the dominant dimension of $B$ is at least two then also $B$ is a Morita algebra.
\end{conjecture}
In \cite{CX} several special cases of this conjecture were proven. In this article we give a counterexample to this conjecture.

\begin{theorem*}
Let $A$ be the Nakayama algebra with Kupisch series [4,5,4,5] with vertices numbered from 0 to 3. Let $M$ be the module $e_0 A \oplus e_1 A \oplus e_3 A \oplus e_1 A/e_1 J^4$. Then $A$ is a Morita algebra and $M$ is a tilting module of projective dimension two such that the algebra $B:=End_A(M)$ is an algebra of dominant dimension equal to 4 that is not a Morita algebra.
\end{theorem*}
Note that $B$ is derived equivalent to $A$, since endomorphism algebras of tilting modules are derived equivalent to the original algebra.
Therefore, our theorem gives a counterexample to the conjecture.
We found the counterexample to the conjecture while experimenting with the GAP-package QPA, see \cite{QPA}.
We thank Hongxing Chen and Changchang Xi for useful discussions in Stuttgart and Changchang Xi for proofreading and useful suggestions.
\section{Proof of the theorem}
In this section we give a proof of the theorem that we group into several smaller lemmas. We assume that the reader is familiar with the basics of the representation theory of finite dimensional algebras as explained for example in \cite{ARS} or \cite{ASS}. We use the conventions of \cite{ASS}. Thus we use right modules and write arrows in quiver algebras from left to right. For background on Nakayama algebras and how to calculate projective or injective resolutions for modules in such algebras we refer to \cite{Mar}.
All algebras will be given by quiver and relations and are connected. 
Recall that the Kupisch series of a Nakayama algebra is just the sequence $[a_0,a_1,...,a_r]$ when $a_i$ denotes the dimension of the indecomposable projective modules corresponding to point $i$.
Let $A$ always be the Nakayama algebra with Kupisch series [4,5,4,5]. Thus $A$ is a quiver algebra with a cyclic quiver. We assume that the vertices are numbered from 0 to 3. The quiver of $A$ looks as follows:
$$
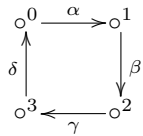
\xymatrix@1{ \circ^{0}\ar [r]^{\alpha}  & \circ^{1}\ar [d]^{\beta} &  &  & & &  &  \\ \circ^{3} \ar [u]^{\delta} & \circ^{2} \ar[l]^{\gamma}}.$$ 
We denote the idempotents corresponding to the points $i$ by $e_i$ and the simple modules corresponding to $i$ by $S_i$.
By $J$ we denote the Jacobson radical of an algebra.
\begin{lemma} \label{lemma 1}
$A$ is a Morita algebra with dominant dimension equal to two.
\end{lemma}
\begin{proof}
The projective-injective indecomposable $A$-modules are $e_1 A$ and $e_3 A$. Thus the minimal faithful projective-injective $A$-module is $eA$ with $e=e_1 + e_3$ and we have that $eAe$ is the symmetric Nakayama algebra with Kupisch series $[3,3]$.
The minimal injective coresolution of $e_0 A$ is as follows:
$$0 \rightarrow e_0 A \rightarrow e_3 A \rightarrow e_3 A \rightarrow e_3 A / e_3 J^4 \rightarrow 0. \ \ (*)$$
As $e_3 A$ is projective-injective and $e_3 A / e_3 J^4$ is not projective but injective, this shows that $e_0A$ has dominant dimension equal to two.
The minimal injective coresolution of $e_2 A$ looks as follows:
$$0 \rightarrow e_2 A \rightarrow e_1 A \rightarrow e_1 A \rightarrow e_1 A / e_1 J^4 \rightarrow 0. \ \ (**)$$
As $e_1 A$ is projective-injective and $e_1 A / e_1 J^4$ is not projective but injective, this shows that $e_2A$ has dominant dimension equal to two.
Since the dominant dimension of an algebra is equal to the minimum of the dominant dimensions of the indecomposable projective modules, we conclude that $A$ has dominant dimension equal to two and thus is a Morita algebra.
\end{proof}

Now let $M:=e_0 A \oplus e_1 A \oplus e_3 A \oplus e_1 A/e_1 J^4$.
Recall that a tilting module is a module $T$ over an algebra $\Lambda$ that has finite projective dimension and $Ext_{\Lambda}^i(T,T)=0$ for all $i>0$ such that the regular module $\Lambda$ has a finite coresolution in $add(T)$.
\begin{lemma}
$M$ is a tilting module of projective dimension two.
\end{lemma}

\begin{proof}
Note that $M$ has three indecomposable projective modules as direct summands where only $e_0A$ is not injective and one indecomposable injective non-projective module, namely $e_1 A/e_1 J^4$.
The following minimal projective resolution of $e_1 A / e_1 J^4$ shows that the projective dimension of $M$ is equal to two:
$$0 \rightarrow e_2 A \rightarrow e_1 A \rightarrow e_1 A \rightarrow e_1 A/ e_1 J^4 \rightarrow 0.$$
Now the exact sequence $(**)$ in the proof of \ref{lemma 1} shows that $A$ has a coresolution in $add(M)$.
What is left to show is that $Ext_A^i(M,M)=0$ for $i=1$ and $i=2$ because $Ext_A^i(M,M)=0$ for $i>2$ since $M$ has projective dimension 2. Note that $\Omega^1(M) = e_1 J^4 \cong S_1$.
We have $Ext_A^1(M,M)=Ext_A^1(e_1 A / e_1 J^4, e_0 A)$ and $Ext_A^2(M,M)=Ext_A^1(\Omega^1(M),M)=Ext_A^1(S_1 , e_0 A)$. Note that in general for a simple module $S$ and a module $N$ over an algebra $\Lambda$, we have $Ext_{\Lambda}^1(S,N)=0$ iff the socle of $I^1(N)$ does not have $S$ as a direct summand when $(I^i(N))$ denotes a minimal injective coresolution of $N$, see for example \cite{Ben} corollary 2.5.4.
This gives us that $Ext_A^1(S_1 , e_0 A)=0$ when looking at the minimal injective coresolution of $e_0 A$ in $(*)$ in the proof of \ref{lemma 1}.
Now we show that $Ext_A^1(e_1 A / e_1 J^4, e_0 A)=0$.
Look at the following short exact sequence:
$$0 \rightarrow e_1 J^4 \rightarrow e_1 A \rightarrow e_1 A / e_1 J^4 \rightarrow 0.$$
We apply the functor $Hom_A(-,e_0 A)$ to this short exact sequence and obtain the following exact sequence:
$$0 \rightarrow Hom_A(e_1 A / e_1 J^4, e_0 A) \rightarrow Hom_A(e_1 A, e_0 A) \rightarrow Hom_A(S_1, e_0 A) \rightarrow Ext_A^1(e_1 A/ e_1 J^4, e_0 A) \rightarrow 0.$$
This gives us that $Ext_A^1(e_1 A/ e_1 J^4, e_0 A)=0$ iff $dim(Hom_A(e_1 A, e_0 A))=dim(Hom_A(e_1 A / e_1 J^4, e_0 A))+dim(Hom_A(S_1, e_0 A))$, which is true since $dim(Hom_A(e_1 A, e_0 A))=1$ and $dim(Hom_A(e_1 A / e_1 J^4, e_0 A))=1$ but $dim(Hom_A(S_1, e_0 A))=0$.
This proves that $Ext_A^i(M,M)=0$ for all $i>0$ and thus that $M$ is a tilting module of projective dimension two.

\end{proof}

Now let $B:=End_A(M)$ be the endomorphism ring of $M$.
\begin{lemma}
$B$ is a Nakayama algebra given by quiver and relations with Kupisch series $[4,4,5,5]$.
\end{lemma}
\begin{proof}

By the main theorem of \cite{Yam2}, the endomorphism ring of a module over a Nakayama algebra which only has indecomposable projective or injective modules as a direct summands is again a Nakayama algebra. Also note that $B$ is a basic algebra since $M$ is a basic module and $B$ has simple modules isomorphic to $End_A(M_i)/rad(End_A(M_i))$, which are one-dimensional modules when $M_i$ denote the indecomposable direct summands of $M$. A basic algebra with all simple modules of dimension equal to one is given by quiver and relations.
We therefore just have to determine the Kupisch series of $B$.
We have 
$$B=End_A(e_0 A \oplus e_1 A \oplus e_3 A \oplus e_1 A/e_1 J^4)=$$
$$\begin{pmatrix}[1]
  e_0 A e_0 & e_0 A e_1 & e_0 A e_3 & e_0 J e_1 \\
  e_1 A e_0 & e_1 A e_1 & e_1 A e_3 & e_1 J e_1 \\
  e_3 A e_0 & e_3 A e_1 & e_3 A e_3 & e_3 J e_1 \\
  (e_1 A/ e_1 J^4)e_0 & (e_1 A/ e_1 J^4)e_1 & (e_1 A/ e_1 J^4)e_3 & (e_1 A/ e_1 J^4)e_1 
 \end{pmatrix} .$$
Noting that $e_0 J e_0=0$ and $(e_1 J/ e_1 J^4)e_1=0$, the radical of $B$ is then equal to 
$$\begin{pmatrix}[1]
  0 & e_0 A e_1 & e_0 A e_3 & e_0 J e_1 \\
  e_1 A e_0 & e_1 J e_1 & e_1 A e_3 & e_1 J e_1 \\
  e_3 A e_0 & e_3 A e_1 & e_3 J e_3 & e_3 J e_1 \\
  (e_1 A/ e_1 J^4)e_0 & (e_1 A/ e_1 J^4)e_1 & (e_1 A/ e_1 J^4)e_3 & 0
 \end{pmatrix} .$$
 We have $rad^2(B)=rad(B) rad(B)$ and multiplication gives that the (1,4)-entry of $rad^2(B)$ is equal to $e_0 A e_1 J e_1 +e_0 A e_3 J e_1=0$.
 Thus the (1,4)-entry in $rad(B)/rad^2(B)$ is $e_0 J e_1$, which is non-zero.
 This gives us that there is an arrow in the quiver of $B$ from the first point to the fourth point. Now the projective indecomposable $B$-modules are given by $Hom_A(M,M_i)$. We have $dim(Hom_A(M,e_0A))=4, dim(Hom_A(M,e_1 A))=5 , dim(Hom_A(M,e_3A))=5 $ and $dim(Hom_A(M,e_1A/e_1 J^4))=4$ and we saw that there is an arrow in the quiver of $B$ from a point whose corresponding indecomposable projective module has dimension 4 and a point whose corresponding indecomposable projective module has dimension 4. This gives us that the Kupisch series can only be $[4,4,5,5]$.
\end{proof}

After renumbering we can assume that he quiver of the Nakayama algebra $B$ with Kupisch series [4,4,5,5] looks as follows:
$$\xymatrix@1{ \circ^{0}\ar [r]^{a}  & \circ^{1}\ar [d]^{b} &  &  & & &  &  \\ \circ^{3} \ar [u]^{d} & \circ^{2} \ar[l]^{c}}.$$ 
\begin{lemma}
$B$ has dominant dimension equal to 4 but is not a Morita algebra.
\end{lemma}

\begin{proof}
The projective-injective indecomposable $B$-modules are $e_1 B, e_2 B$ and $e_3B$.   Thus the minimal faithful projective-injective $B$-module is $eB$ with $e=e_1 +e_2 +e_3$. The algebra $eBe$ is the Nakayama algebra with Kupisch series $[3,4,4]$, which is not selfinjective.
What is left to show it that $B$ has dominant dimension equal to 4.
We give the minimal injective coresolution of $e_0 B$:
$$0 \rightarrow e_0 B \rightarrow e_3 B \rightarrow e_3 B \rightarrow e_2 B \rightarrow e_2 B \rightarrow D(B e_1) \rightarrow 0.$$
This shows that $e_0 B$ has dominant dimension equal to 4 and also that $B$ has dominant dimension equal to 4 since the dominant dimension of the regular module equals the minimum of the dominant dimensions of the indecomposable projective modules.
\end{proof}

Combining all the results of this section we obtain the following theorem:
\begin{theorem}
Let $A$ be the Nakayama algebra with Kupisch series [4,5,4,5] with vertices numbered from 0 to 3. Let $M$ be the module $e_0 A \oplus e_1 A \oplus e_3 A \oplus e_1 A/e_1 J^4$. Then $A$ is a Morita algebra and $M$ is a tilting module of projective dimension two such that the algebra $B:=End_A(M)$ is an algebra of dominant dimension equal to 4 that is not a Morita algebra.
\end{theorem}

\end{document}